\documentclass[12pt]{article}
\usepackage{amsthm}
\usepackage{amssymb,latexsym,amsmath}
\usepackage{amscd}
\usepackage{amsthm,amssymb,amsmath}
\usepackage[all]{xy}
\usepackage[margin=1in]{geometry}
\usepackage[colorlinks]{hyperref}
\usepackage{pgf,tikz}
\usepackage{tkz-graph}
\usepackage{subfig}
\usepackage{float}

\oddsidemargin=\evensidemargin \footskip=36pt \voffset=-1.5cm
\DeclareMathOperator{\aut}{Aut}

 \DeclareMathOperator{\sym}{Sym}

\def\tm#1{\item[{\rm (#1)}]}
\def\css{\begin{cases}}
\def\ecss{\end{cases}}

\def\nmrt{\begin{enumerate}}
\def\enmrt{\end{enumerate}}
\def\tm#1{\item[{\rm (#1)}]}

\makeatother

\newtheorem{formula}{}[section]
\newtheorem{proposition}[formula]{Proposition}
\newtheorem{definition}[formula]{Definition}
\newtheorem{corollary}[formula]{Corollary}

\newtheorem{lemma}[formula]{Lemma}
\newtheorem{theorem}[formula]{Theorem}

\newtheorem{example}[formula]{Example}

\begin{document}

\title{The generalized X-join of Cayley graphs}
\author{
\ Allen Herman\thanks{The work of the first author was supported by a grant from NSERC.},  Javad Bagherian, and Hanieh Memarzadeh
}
\maketitle

\begin{abstract}
As a main result of this paper we give conditions under which  the generalized $X$-join of Cayley graphs is a Cayley graph. In particular, we show that $X$-join of isomorphic Cayley graphs is a Cayley grpah. To do this, new properties for a generalized wreath product of permutation groups are given in the case where the base group acts regularly.  These are used to give conditions for the generalized wreath product to contain a regular subgroup, which are then applied to generalized $X$-joins of Cayley graphs.  Along the way, it is shown that the generalized $X$-join of isomorphic Cayley graphs will always be a vertex transitive graph.
\end{abstract}
\smallskip



\section{Introduction}

In this article we investigate properties of the generalized $X$-join of graphs introduced in \cite{BM}.  This graph product generalizes the $X$-join of graphs originally introduced by Sabidussi and Hemminger in \cite{Sab3} and \cite{H}.  Given a graph $G$ with vertex set $\{1,\dots,n\}$, and $n$ graphs $H_1, \dots, H_n$, the {\it $G$-join} \;$G[H_1,\dots,H_n]$ is the graph that extends the disjoint union of the graphs $H_1,\dots,H_n$ by adding edges between all vertices of $H_i$ and $H_j$ when $(i,j)$ is an edge in $G$ for $1 \le i,j \le n$. (This construction does not require the graphs to be ordinary, nor finite, and applies equally well to looped graphs or digraphs.) This was treated for ordinary graphs under the name generalized composition by Schwenk \cite{Schw}, where he showed that if the $H_i$ are regular graphs, then the collection $\{H_1,\dots,H_n\}$ gives an equitable partition of $G[H_1,\dots,H_n].$  One recovers the lexicographic product $G \circ H$ in the case where all of the graphs $H_i$ are copies of the same graph $H$. Recently some articles have appeared establishing new properties for generalized compositions of graphs (see, for instance, \cite{CAAR2013}, \cite{TWW}) and studying their applications in hierarchical and computer networks (see \cite{AK} and \cite{G2010}).

The definition of the generalized $X$-join $G \circ_{\lambda} \{H_1,\dots,H_n\}$ we investigate here is a modification of the one for generalized composition, the latter corresponding to the case where the elements of the defining partition of $G$ into $n$ disjoint subsets are all singletons (see \cite[Definition 2]{AK}).  The definition is quite flexible compared to the lexicographic product or the $X$-join, as it applies to many other graphs. In \cite{BM}, the last two authors considered the question of how to recognize a graph as a generalized $X$-join of graphs. They showed that the automorphism group of $G \circ_{\lambda} \{H_1,\dots,H_n\}$ contains the generalized wreath product $\aut(G) \circ_{\lambda} \{\aut(H_1),\dots,\aut(H_n)\}$ as a subgroup, and gave conditions for these to be equal.  (For the definition of the generalized wreath product of groups, see \cite{Ba}, \cite{ip3}.) Their result generalizes several earlier results for the automorphism group of lexicographic products \cite{DM}, \cite{Sab1}, \cite{Sab2}, and the $X$-join \cite{H}.

Graphs that are the direct, lexicographic, wreathed, or generalized wreathed product of Cayley graphs have been shown to be a Cayley graph (see \cite{D}, \cite{hand}, and \cite{Fen}), but we do not know if this will hold for the generalized $X$-join of Cayley graphs.
Note that it is also  unknown  for $X$-join of Cayley graphs.
The main result of this article gives conditions for the generalized $X$-join of Cayley graphs to be a Cayley graph.
In particular, it shows that the $X$-join of isomorphic Cayley graphs is a Cayley graph.

In Section 2 we review the definitions of the generalized $X$-join of graphs and generalized wreath product of groups from $\cite{BM}$ and explore conditions where a generalized wreath product of groups contains a regular subgroup in a natural way.  In Section 3 we study properties of the generalized $X$-join of Cayley graphs.  When $W$ is the generalized $X$-join of Cayley graphs defined over the groups $A$ and $C_X$ for $X \in \Sigma$, we find a natural condition under which the generalized wreath product $A \circ_{\lambda} \{C_X\}_{X \in \Sigma }$ will be a subgroup of $\aut(W)$ that acts transitively on vertices.   We then give conditions for the generalized $X$-join of Cayley graphs to be a Cayley graph, and conclude with two examples.

We remark that, at least in spirit, the generalized $X$-join construction discussed here is already being applied in the ``lifting Hamilton cycles" attacks on Lovaczs' conjecture that every finite Cayley graph should have a Hamilton cycle (see \cite[pg.~5493]{KM}).   Indeed, it is relatively easy to see that a graph which is a generalized $X$-join of graphs that have Hamilton cycles will have a Hamilton cycle itself.

\section{Preliminaries}

Throughout this paper, $G = (V,E)$ will denote an ordinary graph; i.e.~a finite undirected graph with no multiple edges whose vertex set is $V = V (G)$ and edge set is $E = E(G)$.  If all pairs of vertices of a subgraph $H$ of $G$
that are adjacent in $G$ are also adjacent in $H$, then $H$ is an induced subgraph.  For $X \subseteq V$ we write $G[X]$ for the subgraph of $G$ induced by $X$.

For two graphs $G=(V,E)$ and $H=(V',E')$, a graph homomorphism $f:G\rightarrow H$ is a mapping $f:V\rightarrow V'$ such that $(f(u),f(v))\in E'$ whenever $(u,v)\in E$.   If $f:V\rightarrow V'$ is a bijection and $f^{-1}:H\rightarrow G$ is also a graph homomorphism, then $f:G\rightarrow H$ is called a graph isomorphism, and we say that $G$ and $H$ are isomorphic.

For a finite set $V$, the group of all permutations of $V$ is denoted by $\sym(V)$.  The automorphism group of a graph $G=(V,E)$ is the subgroup $\aut(G)$ of $\sym(V)$ consisting of the permutations of $V$ that induce automorphisms of the graph $G$.   Every subgroup of $\sym(V)$ is called a permutation group on $V$.  For $F \leq \sym(V)$ and $X\subseteq V$, the setwise stabilizer of $X$ in $F$ is $\{f \in F : X^f=X\}$ and the pointwise stabilizer of $X$ in $F$ is $F_{(X)}=\{f\in F : x^f=x, \, \forall x\in X\}.$  By a system of blocks $\Sigma$ for a permutation group $F\leq \sym(\Omega)$ we mean
\nmrt
\tm{1} $\Sigma$ is a partition of $\Omega$; and
\tm{2} for every $X \in \Sigma$ and every $f\in F$, $X^f\cap X=\varnothing $ or $ X^f=X.$
\enmrt
If $\Sigma$ is a system of blocks for $F$ and $X\in \Sigma$, we write $F^X$ for the setwise stabilizer of $X$ in $F$, viewed as a permutation group on $X$.  For other terminology and notation not defined here, see \cite{BM}.

\medskip
\noindent {\bf 2.1 The generalized $X$-join of graphs}

\begin{definition}
Let $G=(V,R)$ be a graph and $\Sigma$ be a partition of $V$. Suppose that for every $X\in \Sigma$ there exist a graph $ B_X=(Y_X,E_X)$
and a graph epimorphism $\lambda_X:Y_X\rightarrow X$ from $B_X$ onto $G[X]$. Put $Y=\bigcup _{X\in \Sigma}Y_X$ and $\lambda=\bigcup_{X\in \Sigma}\lambda_X$. We define a graph $W$ with vertex set $Y$ and edge set $E$ such that $(y,y')\in E$ if
\nmrt
\tm{1} either $(y,y')\in E_X$, for some $X\in \Sigma$;
\tm{2} or $(y,y')\in \lambda _X^{-1}(x)\times \lambda _{X'}^{-1}(x')$ where $ X\neq X'$ and $(x,x')\in R$.
\enmrt
We call the graph $W=(Y,E)$ the {\it generalized X-join} of $G$ and $\{B_X\} _{X\in \Sigma }$ with respect to $\lambda$, and we denote it by $G\circ_\lambda\{B_X\}_{X\in\Sigma}$.
\end{definition}

\begin{example} {\rm
Suppose that $G$,  $B_X$ and $ B_{X'}$ be the graphs in Figure 1 with vertex sets $V=\{1,2,3,4\}$, $Y_X=\{a,b,c,d\}$ and $Y_{X'}=\{a',b',c',d'\}$, respectively.

\begin{figure}[htb]
\begin{center}
\begin{tikzpicture}
\filldraw (-1,4) circle [radius=0.04] node  [above] {$1$} -- (1,4) circle [radius=0.04] node  [above] {$2$} -- (1,3) circle [radius=0.04] node  [below] {$3$} -- (-1,3) circle [radius=0.04] node  [below] {$4$} -- (-1,4) circle [radius=0.04] node  [above] {$1$};
\filldraw (1,4) circle [radius=0.04] node  [above] {$2$} -- (-1,3) circle [radius=0.04] node  [below] {$4$};
\end{tikzpicture}
\subfloat{$G$}\;\;\;\;\;\;
\begin{tikzpicture}
\filldraw (3,4) circle [radius=0.04] node  [above] {$a$} -- (5,3) circle [radius=0.04] node  [below] {$b$};
 \filldraw (5,4) circle [radius=0.04] node  [above] {$c$} -- (3,3) circle [radius=0.04] node  [below] {$d$};
\end{tikzpicture}\;
\subfloat{$B_X$}\;\;\;\;\;\;
\begin{tikzpicture}
\filldraw (9,4) circle [radius=0.04] node  [above] {$e$}  -- (9,3) circle [radius=0.04] node [below] {$g$} -- (7,3) circle [radius=0.04] node  [below] {$f$};
\filldraw (9,4) circle [radius=0.04] node  [above] {$e$} -- (7,3) circle [radius=0.04] node  [below] {$f$};
\end{tikzpicture}
\subfloat{$B_{X'}$}\;\;\;\;\;
\end{center}
\caption{The graphs $G$, $B_X$, and $B_{X'}$}
\end{figure}
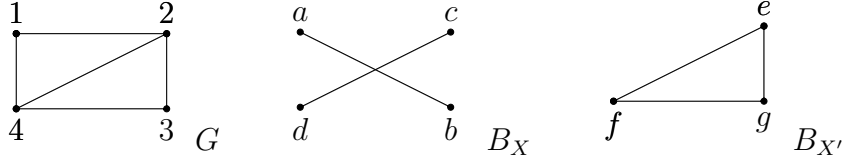

\begin{figure}[htb]
\begin{center}
\begin{tikzpicture}
\filldraw (7,-1.5) circle [radius=0.04] node  [right] {$d$} -- (3,-1.5) circle [radius=0.04] node  [left] {$c$};
\filldraw (1,-1.5) circle [radius=0.04] node [right] {$b$};
\filldraw (0,1) circle [radius=0.04] node [above] {$f$} -- (4,1) circle [radius=0.04] node [above] {$g$};
\filldraw (0,1) circle [radius=0.04] node [above] {$f$} -- (-3,-1.5) circle [radius=0.04] node [left] {$a$} -- (4,1) circle [radius=0.04] node  [above] {$g$} -- (1,-1.5) circle [radius=0.04] node [right] {$b$} -- (0,1) circle [radius=0.04] node  [above] {$f$} -- (3,-1.5) circle [radius=0.04] node [left] {$c$} -- (4,1) circle [radius=0.04] node  [above] {$g$} -- (7,-1.5) circle [radius=0.04] node [right] {$d$} -- (0,1) circle [radius=0.04] node [above] {$f$};
\filldraw (-3,-1.5) circle [radius=0.04] node [left] {$a$} -- (1,-1.5) circle [radius=0.04] node  [right] {$b$};
\filldraw (4,1) circle [radius=0.04] node  [above] {$g$} -- (2,2) circle [radius=0.04] node [above] {$e$} -- (0,1) circle [radius=0.04] node [above] {$f$};
\filldraw (2,2) circle [radius=0.04] node [above] {$e$} -- (-3,-1.5) circle [radius=0.04] node [left] {$a$};
\filldraw (2,2) circle [radius=0.04] node [above] {$e$} -- (1,-1.5) circle [radius=0.04] node [right] {$b$};
\filldraw (2,2) circle [radius=0.04] node [above] {$e$} -- (3,-1.5) circle [radius=0.04] node [left] {$c$};
\filldraw (2,2) circle [radius=0.04] node [above] {$e$} -- (7,-1.5) circle [radius=0.04] node [right] {$d$};
\end{tikzpicture}
\end{center}
\caption{The graph $W=G\circ_\lambda\{B_X\}_{X\in\Sigma}$}
\end{figure}
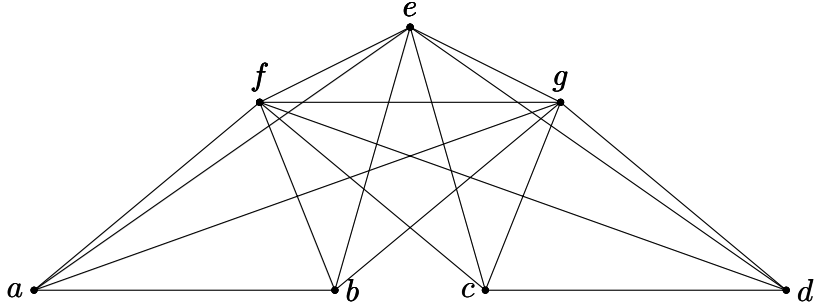

Consider the partition $\Sigma=\{X=\{1,3\},X'=\{2,4\}\}$ of $V$. Suppose the graph epimorphisms  $\lambda_X:B_X\rightarrow G[X]$ and $\lambda_{X'}:B_{X'}\rightarrow G[X']$ are given by the following:
$$\begin{array}{rclcrcl}
\lambda_X(a) &=& \lambda_X(b) \; = \; 1 & \qquad \qquad & \lambda_{X'}(e) &=& 2 \\
\lambda_X(c)&=&\lambda_X(d) \; = \; 3 & & \lambda_{X'}(f) &=& \lambda_{X'}(g) \; = \; 4.
\end{array}$$
Then the generalized X-join of $G$ and $\{B_X, B_{X'}\}$ with respect to $\lambda$ is the graph in Figure 2.
}
\end{example}

\medskip
\noindent {\bf 2.2 Generalized wreath products with a regular base group}

\medskip
Here we review the generalized wreath product construction from \cite{Ba} in the special case where the base group $F$ acts regularly.  Up to permutation isomorphism we can assume $F$ acts regularly on itself by right multiplication.  If $\Sigma$ is a system of blocks for this action, then for a fixed $X \in \Sigma$ and $x_0 \in X$, we will have that $H = \{ f \in F : x_0f \in X \}$ is a subgroup of $F$, and the map $x_0f \mapsto f$ for $f \in H$ gives a bijection from $X$ onto $H$.  For all $X' \in \Sigma$, we can choose an $f_{X'} \in F$ with $X' = Xf_{X'}$, and set $f_X = 1_F$ for convenience.  Let $\mathcal{S} = \{f_{X'} : X' \in \Sigma \}$ be this fixed set of right coset representatives for $H$ in $F$.  With these choices, the setwise stabilizer $F^{X'} = \{ f \in F : X'f = X' \}$ can be identified with the conjugate subgroup $H^{f_{X'}}$ for all $X' \in \Sigma$, and so conjugation by $f_{X'} \in \mathcal{S}$ defines an isomorphism $\mu_{XX'}: F^{X} \rightarrow F^{X'}$.  Also, for each $X' \in \Sigma$, right multiplication by $f_{X'} \in \mathcal{S}$ defines a bijection $\varepsilon_{XX'}: X \rightarrow X'$.  Given any two elements $X',X'' \in \Sigma$, we let $\mu_{X'X''}:F^{X'} \rightarrow F^{X''}$ be conjugation by $f_{X'}^{-1} f_{X''}$ and $\varepsilon_{X'X''}:X' \rightarrow X''$ to be equal to right multiplication by $f_{X'}^{-1}f_{X''}$.

Suppose $\{M_{X'}: X' \in \Sigma\}$ is a set of pairwise isomorphic groups indexed by $\Sigma$.  For our fixed $X \in \Sigma$, let $\theta_X: M_X \rightarrow H$ be a fixed group epimorphism with kernel $K_X$.   Since $\theta_X: M_X \rightarrow H$ is an epimorphism, for all $h \in H$ we can choose $t_h \in M_X$ such that $\theta_X(t_h) = h$, and set $t_{1_F} = 1_{M_X}$.  The set $T_H = \{t_h : h \in H\}$ will be a set of right coset representatives for $K_X$ in $M_X$.  For each $X' \in \Sigma$, we fix a group isomorphism $\psi_{XX'}: M_X \rightarrow M_{X'}$.   Then we can define a group epimorphism $\theta_{X'}:M_{X'} \rightarrow F^{X'}=H^{f_{X'}}$ by setting $\theta_{X'}(m') = \mu_{XX'}(\theta_X(\psi_{XX'}^{-1}(m')))$ for all $m' \in M_{X'}$.  Then these maps have the property that for all $m \in M_X$,
$$ \theta_{X'}(\psi_{XX'}(m))=\mu_{XX'}(\theta_X(m)). $$
Given any pair $X',X'' \in \Sigma$, the map $\psi_{X'X''}=\psi_{XX''} \circ \psi_{XX'}^{-1}$ will satisfy
$$ \theta_{X''}(\psi_{X'X''}(m')) = \mu_{X'X''}(\theta_{X'}(m')), \quad \mbox{ for all } m' \in M_{X'}. $$
For each $X' \in \Sigma$, let $K_{X'}$ be the kernel of $\theta_{X'}$, so $M_{X'}/K_{X'} \simeq F^{X'}$ via the canonical isomorphism.

Let $Y_X = M_X$ as sets, and let $M_X$ act on $Y_X$ by right multiplication.  Let $\lambda_X: Y_X \rightarrow X$ be defined by $\lambda_X(y) = \theta_X(y)$ for all $y \in Y_X$. This map has the property that for all $y \in Y_X$ and $m \in M_X$, $\lambda_X(ym) = \lambda_X(y)\theta_X(m)$.  For each $X' = Hf_{X'}$, let $Y_{X'}$ be a copy of $Y_{X}$.  Let $\phi_{XX'}: Y_X \rightarrow Y_{X'}$ be a fixed bijection.  Then there is a unique surjection $\lambda_{X'}:Y_{X'} \rightarrow X' = Hf_{X'}$ satisfying
$$ \lambda_{X'}(\phi_{XX'}(y))=\varepsilon_{XX'}(\lambda_X(y)), \mbox{ for all } y \in Y_X.$$
We can define an action of $M_{X'}$ on $Y_{X'}$ such that for all $y \in Y_X$ and $m \in M_X$,
$$\phi_{XX'}(y)\psi_{XX'}(m)=\phi_{XX'}(ym).$$
For every pair $X',X'' \in \Sigma$, we set $\phi_{X'X''}:Y_{X'} \rightarrow Y_{X''}$ to be the bijection $\phi_{X'X''} = \phi_{XX''}\circ \phi_{XX'}^{-1}$.  These bijections satisfy the compatibility relation $\lambda_{X''}(\phi_{X'X''}(y')) = \varepsilon_{X'X''}(\lambda_{X'}(y'))$ for all $y' \in Y_{X'}$, and furthermore, for all $y' \in Y_{X'}$ and $m' \in M_{X'}$, $\phi_{X'X''}(y')\psi_{X'X''}(m') = \phi_{X'X''}(y'm').$

Now, let $Y = \displaystyle{\bigcup_{X' \in \Sigma}} Y_{X'}$, and let $K = \prod_{X' \in \Sigma} K_{X'}$.  We view $K$ as a subgroup of $Sym(Y)$ in the following way:  if $k = \prod_{X'} k_{X'}$ for some independent choices $k_{X'} \in K_{X'}$ for $X' \in \Sigma$, then $y \cdot k = y k_{X'}$ whenever $y \in Y$ belongs to $Y_{X'}$.  As a subgroup of $Sym(Y)$, $K$ contains the natural {\it diagonal} subgroup $J = \{ \prod_{X' \in \Sigma} \psi_{XX'}(\ell) : \ell \in K_X \}$. For $\ell \in K_X$, we let $\hat{\ell} = \prod_{X'} \psi_{XX'} (\ell) \in J$.

For each $f \in F$, we associate $f$ to a permutation $\overline{f} \in Sym(Y)$ that lifts right multiplication by $f$ as a map from $F$ to $F$.  Given $y \in Y_{X'} \subset Y$, either $X' f = X'$ or $X' f = X'' \ne X'$.  If $X' f = X'$, then $f_{X'} f = h f_{X'}$ for some $h \in H$.  In this case $y \overline{f}$ is defined to be $y \overline{f} = y \psi_{XX'}(t_h)$.   If $X' f = X'' \ne X'$, then $f_{X'} f = h f_{X''}$ for some $h \in H$, so $f = h^{f_{X'}} f_{X'}^{-1} f_{X''}$.  So in this case we define $y \overline{f} = \phi_{X'X''}(y \psi_{XX'}(t_h))$.
(This definition is equivalent to that of \cite{Ba} in the case where the base group $F$ has a regular action.)  We remark that for each $f \in F$, the definition of $\bar{f}$ is dependent upon the choice of coset representatives for the elements of $T_H$, so in fact we have a precise definition of $\bar{f}$ only up to composition by an element of $K$.  Intuitively, think of $Y_{X'}$ as $Y_X \overline{f_{X'}}$, and think of right multiplication by $\overline{f_{X'}}$ for $f_{X'} \in \mathcal{S}$ as the map $\phi_{XX'}$, since this lifts right multiplication $H \rightarrow Hf_{X'}$ to a map $Y_X \rightarrow Y_{X'}$.  When $f_{X'}f = h f_{X''}$ for $X'' \in \Sigma$, then the restriction of $\overline{f}$ to $Y_{X'}$ will take the same image on $Y_{X'}$ as $\overline{f_{X'}}^{-1} \overline{h} \,  \overline{f_{X''}}$.  We can also intuitively think of $M_{X'}$ as the conjugate $(M_X)^{\overline{f_{X'}}}$, so that the identity $\phi_{XX'}(y) \psi_{XX'}(m) = \phi_{XX'}(ym)$ for $y \in Y_X$ and $m \in M_X$ becomes $(y \overline{f_{X'}}) (m^{\overline{f_{X'}}}) = (ym) \overline{f_{X'}}$.

Let $\overline{F} = \{ \overline{f} : f \in F \}$.  $\overline{F}$ is a subset of $Sym(Y)$, which is a subgroup only under a strict condition.

\begin{lemma} The following are equivalent:
\nmrt
\tm{i} The map $f \mapsto \overline{f}$, for all $f \in F$, is a group isomorphism $F \simeq \overline{F}$.
\tm{ii} $T_H$ is a group.
\tm{iii} $M_X = K_X \rtimes T_H \simeq K_X \rtimes H$.
\enmrt
\end{lemma}

\begin{proof}
 For $f_1,f_2 \in F$, if $y' \in Y_{X'} \subset Y$ is equal to $y \overline{f_{X'}}$ for some $y \in Y_X$, then there exist $h', h'' \in H$ and $f_{X''}, f_{X'''} \in \mathcal{S}$ such that
$$\begin{array}{rcl}
y' \cdot \overline{f_1} \cdot \overline{f_2} &=& (y t_{h'}) \overline{f_{X''}} \cdot \overline{f_2} \mbox{ (if $f_{X'} f_1 = h' f_{X''}$)} \\
&=& (y t_{h'} t_{h''}) \overline{f_{X'''}} \mbox{ ( if $f_{X''} f_2 = h'' f_{X'''}$).}
\end{array}$$
On the other hand,
$$\begin{array}{rcl}
y' \cdot \overline{f_1f_2} = y t_{h'h''} \overline{f_{X'''}}, \mbox{ since $f_{X'}f_1f_2 = h'h''f_{X'''}$.}
\end{array}$$
From this one can see that $\overline{f_1}\, \overline{f_2} = \overline{f_1f_2}$ for all $f_1, f_2 \in F$ if and only if $t_{h'} t_{h"} t_{h'h"}^{-1} = 1$ for all $h', h'' \in H$.  The latter is equivalent to $T_H$ being a group isomorphic to $H$, and this is equivalent to $M_X$ being a split extension of $K_X$ by $T_H$.
\end{proof}

In the notation of the above proof, we can see that the obstruction to $\overline{F}$ being a group in general is:
$$ y' \cdot \overline{f_1}\, \overline{f_2} (\overline{f_1f_2})^{-1} = y' (\overline{f_{X'}})^{-1} \, \overline{h'} \,\overline{h''}  \, (\overline{h'h''})^{-1} \, \overline{f_{X'}}. $$
Since the map used on the right lifts right multiplication by the identity from $F \rightarrow F$, it lies in $K$.  As $K$  is normalized by $\overline{F}$, $K\overline{F}$ is a subgroup of $Sym(Y)$.  This subgroup is precisely the
{\it generalized wreath product} $F \circ_{\lambda} \{ M_{X'} \}_{X' \in \Sigma}$, which was defined in \cite{Ba} to be the subgroup of $Sym(Y)$ generated by $K$ and $\overline{F}$ in the situation where $F$ is not assumed to act regularly.

\begin{proposition}\label{trans}
\nmrt
\tm{i} For all $f \in F$, for all $X' \in \Sigma$, there exists an $X'' \in \Sigma$ such that $Y_{X'} \bar{f} = Y_{X''}$.
\tm{ii} $F \circ_{\lambda} \{M_{X'}\}_{X' \in \Sigma}$ acts transitively on $Y$.
\enmrt
\end{proposition}

\begin{proof}
Let $f \in F$.  For all $y' \in Y_{X'}$, if $y' = \phi_{XX'}(y)$ then $y' = y \overline{f_{X'}}$, where $f_{X'}$ is the unique element of $\mathcal{S}$ for which $X' = H f_{X'}$.  So $Y_{X'}= Y_X \overline{f_{X'}}$.  If $f = f_{X'}^{-1} h f_{X''}$ for some $h \in H$ and $f_{X''} \in \mathcal{S}$, then using our definitions we get $Y_{X'} \bar{f} = Y_X \bar{h} \, \overline{f_{X''}}$.  Since $y \bar{h} = y t_h$ for all $y \in Y_X$, we have $Y_X \bar{h} = Y_X$, so $Y_{X'} \bar{f} = Y_X \overline{f_{X''}} = \phi_{XX''}(Y_X) = Y_{X''}$.  So $(i)$ holds.

Let $y \in Y_X$.  It suffices to show that for all $X' \in \Sigma$ and $y' \in Y_{X'}$, there is a $k \in K$ and $\overline{f} \in \overline{F}$ such that $y k \overline{f} = y'$.  Let $y \in Y_X$ and $y' \in Y_{X'}$ for some $X' \in \Sigma$.  Since $M_X$ acts transitively on $Y_X$ and $\phi_{XX'}$ is a bijection, we can find an $\ell \in K_X$ and $h \in H$ such that $\phi_{XX'}(y \ell t_h) = y'$.  Since $\ell \in K_X$, $\hat{\ell} \in J \subset K$. Letting $f = h f_{X'}$, this means $y \cdot \hat{\ell} \overline{f} = y'$.  This implies $\langle J, \overline{F} \rangle$ is a subgroup of $F \circ_{\lambda} \{M_{X'}\}_{X' \in \Sigma}$ that acts transitively on the set $Y$.  In particular $F \circ_{\lambda} \{M_{X'}\}_{X' \in \Sigma}$ acts transitively on $Y$, this proves $(ii)$.
\end{proof}

We are interested in cases where the elements used for $\bar{F}$ can be chosen so that the subgroup $\langle J, \overline{F} \rangle$ of $\langle K, \overline{F} \rangle$ acts regularly on $Y$.  In order to do so, it must be as small as possible.  The next result gives a sufficient condition for $|\langle J, \overline{F} \rangle| = |K_X||F|=|Y|$.

\begin{theorem}\label{reggwps}
Let $J, \overline{F} \subset F \circ_{\lambda} \{M_{X'}\}_{X \in \Sigma} \le Sym(Y)$ be as above.
If $F = H C_F(H)$ and $M_X = K_X C_{M_X}(K_X)$, and either $F = H \times F_1$ for some subgroup $F_1$ of $F$ or $M_X \simeq K_X \times H$ then $J\overline{F} = \overline{F}J$ is a subgroup of $Sym(Y)$ of order $|J||F|$ that acts regularly on $Y$.
\end{theorem}

\begin{proof}
Assume suitable choices have been made defining the elements of $\bar{F}$.  From our observation preceding the statement of the theorem, it suffices to show $(i)$ that $(\overline{F})^2 \subseteq J\overline{F}$, and $(ii)$ that the set $\overline{F}$ normalizes $J$, as this will show $J\overline{F}=\overline{F}J$ will be a group of the correct order.

\begin{enumerate}
\tm{i}  First suppose $M_X \simeq K_X \times H$.  Since this implies the elements of $T_H$ can be chosen to form a group, it follows that the elements of $\overline{F}$ can be chosen to form a group, and so (i) will hold.  So assume $F=H \times F_1$ for some subgroup $F_1 \subseteq C_F(H)$, $M_X = K_X C_{M_X}(K_X)$, and that $M_X$ is not isomorphic to $K_X \times H$.   In this case $\overline{F}$ may not be a group, so recall the calculation of the obstruction.   Given $f_1, f_2 \in F$, then for $y' = y \overline{f_{X'}} \in Y_{X'} \subset Y$ with $y \in Y_X$,
$$ y' \overline{f_1} \, \overline{f_2} \, (\overline{f_1f_2})^{-1} = y' (\overline{f_{X'}})^{-1} \, \overline{h'} \, \overline{h''} \, (\overline{h'h''})^{-1} \overline{f_{X'}}, $$
where $f_{X'}f_1 = h'f_{X''}$ and $f_{X''} f_2 = h'' f_{X'''}$ for $h', h'' \in H$ and $f_{X''}, f_{X'''} \in \mathcal{S}$.
In order for the map on the right to lie in $J$, it must be equal to an $\hat{\ell}$ for some $\ell \in K_X$, and this means it cannot depend on the choice of $X'$.  However, the choice of $h'$ and subsequently $h''$ will depend on the choice of $X'$, unless we are in the case where the choices of $h'$ and $h''$ are not affected by $f_{X'}$.  This requires $\mathcal{S}$ to be a subgroup of $F$ that centralizes $H$, since then
$f_{X'} f_1 f_2 = h' f_{X''} f_2 = h' h'' f_{X'''}$ and the elements $h'$ and $h''$ that appear do not depend on $f_{X'}$.  So in the case $M_X = K_X C_{M_X}(K_X)$ and $F=H \times F_1$ for some subgroup $F_1 \subseteq C_F(H)$, (i) will hold when we choose $\mathcal{S}=F_1$.

\tm{ii} When $\mathcal{S}$ centralizes $H$, then for all $h \in H$ and all $y' \in Y_{X'}$ with $y' = y \overline{f_{X'}}$ for some $y \in Y_X$, we have $f_{X'}h = hf_{X'}$, so
$$ y' \cdot \overline{h} = y \overline{f_{X'}} \, \overline{h} = y \overline{h} \, \overline{f_{X'}} $$
$$ = (y t_h) \overline{f_{X'}} = y \overline{f_{X'}} \psi_{XX'}(t_h) = y' \psi_{XX'}(t_h). $$
(When $f_{X'}$ and $h$ do not commute, $y' \overline{h}$ will not be equal to $y' \psi_{XX'}(t_h)$.)

Similarly, the assumption $M_X=K_X C_{M_X}(K_X)$ implies that we can choose the set of coset representatives $T_H$ for $K_X$ in $M_X$ with $t_h \in C_{M_X}(K_X)$ for all $h \in H$.

Let $k \in K_X$, and let $\hat{k} = \prod_{X'} \psi_{XX'}(k)$ be the corresponding element of $J$.   Let $f \in F$, with $f = h f_{X''}$ for $h \in H$ and $f_{X''} \in \mathcal{S}$.  Let $y' \in Y$ and assume $y' = \phi_{XX'}(y) \in Y_{X'}$ for some $y \in Y_X$.   Let $h' \in H$ and $f_{X'''} \in \mathcal{S}$ be such that $f_{X'}f = f_{X'} h f_{X''} = h h' f_{X'''}$.  Then
$$\begin{array}{rcl}
y' \cdot \hat{k} \cdot \overline{f} &=& \phi_{XX'}(y) \hat{k} \cdot \overline{f} \\
&=& \phi_{XX'}(y) \psi_{XX'}(k) \overline{f} \\
&=& \phi_{XX'}(y k) \overline{f} \\
&=& (yk) \overline{f_{X'}} \, \overline{f} \\
&=& (yk) \overline{h h'} \, \overline{f_{X'''}} \quad \mbox{ (since $f_{X'} f = hh'f_{X'''}$)} \\
&=& (y k t_{hh'}) \overline{f_{X'''}} \\
&=& \phi_{XX'''}(ykt_{hh'}) \\
&=& \phi_{XX'''}(y t_{hh'} k) \quad \mbox{ (since $t_{hh'}$ commutes with $k$)} \\
&=& \phi_{XX'''}(y t_{hh'}) \psi_{XX'''}(k) \\
&=& (y t_{hh'}) \overline{f_{X'''}} \hat{k} \\
&=& y \overline{hh'} \, \overline{f_{X'''}} \hat{k} \\
&=& y' \cdot \overline{f} \hat{k}.
\end{array}$$
Therefore, when $\mathcal{S} \subseteq C_F(H)$ and $T_H \subseteq C_{M_X}(K_X)$, $\overline{F}$ normalizes $J$.  This implies $J \overline{F} = \overline{F}J$ as subsets of $Sym(Y)$, so (ii) holds.  The proof of the theorem is complete.
\end{enumerate}
\end{proof}

\section{The generalized $X$-join of Cayley graphs}

Recall the definition of a Cayley graph of a finite group with respect to a symmetric set.

\begin{definition}
Let $A$ be a finite group and let $S$ be a subset of $A$ that is closed under inverses and does not contain the identity. Then the Cayley graph $Cay(A,S)$ is the graph with vertex set $A$ and edge set $E(Cay(A,S))=\{(a,b):ab^{-1}\in S\}.$
\end{definition}

Let $G = Cay(A,S_A)$ be a Cayley graph, with vertices identified with elements of $A$ and $(a,b) \in E(G) \iff ab^{-1} \in S_A$.   Then $A$ acts regularly on $G$, with right multiplications by elements of $A$ correspond to graph automorphisms of $G$.   It is well known that a Cayley graph is vertex-transitive, but the converse need not be true. A graph $G$ is isomorphic to a Cayley graph $Cay(A,S)$ if and only if $\aut(G)$ contains a subgroup $A_0$ that is regular as a permutation group on $V(G).$ In this case $A=A_0$. See \cite[section 10]{hand}.

\begin{proposition}\label{subgp}
Let $G=Cay(A,S_A)$ be a Cayley graph and suppose $\Sigma$ is a system of blocks for the action of $A$ by right multiplication on itself.  For a fixed $X_1 \in \Sigma$, let $H = \{ h \in A : X_1h = X_1 \}$ and let $\{a_X : X \in \Sigma \}$ be a set of right coset representatives for $H$ in $A$.
Suppose that $\{B_X=Cay(C_X,S_X)\}_{X \in \Sigma }$ is a set of pairwise isomorphic Cayley graphs, and assume $\theta_{X_1}: C_{X_1} \rightarrow H$ is a group epimorphism whose values agree with those of a graph epimorphism $\lambda_{X_1}: B_{X_1} \rightarrow G[X_1]$.  Let $Y = \bigcup_{X \in \Sigma} Y_X$, where $Y_X = Y_{X_1} \overline{a_{X}}$ is a copy of $B_{X_1}$ for which $(y \mapsto y \overline{a_X})$, $y \in Y_{X_1}$, is a graph isomorphism.  For all $X \in \Sigma$, let $\lambda_X: Y_X \rightarrow Ha_X$ be the graph epimorphism given by $\lambda_X(y\overline{a_X}) = \lambda_{X_1}(y) a_X$, for all $y \in Y_{X_1}$.  Let $\lambda = \{ \lambda_X \}_{X \in \Sigma}$ be this family of graph epimorphisms.
Then the generalized wreath product $A \circ_{\lambda} \{C_X \}_{X \in \Sigma}$ is a subgroup of the automorphism group of the generalized $X$-join graph $W = G \circ_{\lambda} \{B_X \}_{X \in \Sigma}$.
\end{proposition}

\begin{proof}
By the above argument we can consider $A\circ_\lambda \{C_X\} _{X\in \Sigma}=\langle K,\overline{A}\rangle$ where $K= \prod_{X' \in \Sigma} {K_{X'}}$. We let ${k_{X'}}\in K_{X'}$ for all $X' \in \Sigma.$
Suppose $(u', v'') \in E(W)$, with $u' \in Y_{X'}$ and $v'' \in Y_{X''}$ for $X', X'' \in \Sigma$.  Let $u,v \in Y_{X_1}$ with $u' = u \overline{a_{X'}}$ and $v'' = v \overline{a_{X''}}$.  Then
$(u' k  , v'' k ) = (u' k_{X'}, v'' k_{X''})$.  If $X' \ne X''$, then
$$\begin{array}{rcl}
(u',v'') \in E(W) &\iff& (\lambda_{X'}(u'),\lambda_{X''}(v'')) \in E(G). \\
\end{array}$$
Since $K_{X'} = \ker \theta_{X'}$ for all $X' \in \Sigma$,
$$\begin{array}{rcl}
(\lambda_{X'}(u'k), \lambda_{X''}(v''k )) &=& (\lambda_{X'}(u'k_{X'}), \lambda_{X''}(v''k_{X''})) \\ &=& (\lambda_{X'}(u') \theta_{X'}(k_{X'}), \lambda_{X''}(v'')\theta_{X''}(k_{X''})) \\
&=& (\lambda_{X'}(u'), \lambda_{X''}(v'')),
\end{array}$$
so $(u'k , v''k ) \in E(W)$.  If $X' = X''$, then $(u',v'') = (u\overline{a_{X'}},v\overline{a_{X'}}) = (\phi_{{X_1}X'}(u), \phi_{{X_1}X'}(v))$. So
$$\begin{array}{rcl}
(u',v'') \in E(W) &\iff& (u',v'') \in E(B_{X'}) \\
&\iff& (\phi_{{X_1}X'}(u),\phi_{{X_1}X'}(v)) \in E(B_{X'}) \\
&\iff& (u,v) \in E(B_{X_1}) \quad \mbox{ (since $\phi_{{X_1}X'}$ is a graph isomorphism).} \\
\end{array}$$
Now, we have that
$$\begin{array}{rcl}
(u'k ,v''k ) &=& (\phi_{{X_1}X'}(u) k_{X'},\phi_{{X_1}X'}(v) k_{X'}).
\end{array}$$
Since $\phi_{{X_1}X'}$ and right multiplication by $k_{X'} \in C_{X'}$ are graph isomorphisms,
$$\begin{array}{rcl}
(u'k ,v''k ) \in E(W) &\iff& (\phi_{{X_1}X'}(u) k_{X'},\phi_{{X_1}X'}(v) k_{X'})\in E(B_{X'}) \\
&\iff& (\phi_{{X_1}X'}(u),\phi_{{X_1}X'}(v))\in E(B_{X'}) \\
&\iff& (u,v) \in E(B_{X_1}), \\
\end{array}$$
thus $(u',v'') \in E(W) \iff (u'k ,v''k) \in E(W)$.

Next, let $\overline{a} \in \overline{A}$.  Suppose $(u',v'') \in E(W)$, with $u' \in Y_{X'}$ and $v'' \in Y_{X''}$.  Let $u,v \in Y_{X_1}$ be such that $u' = \phi_{{X_1}X'}(u)$ and $v'' = \phi_{{X_1}X''}(v)$.  If $X' \ne X''$, then
$$\begin{array}{rcl}
(u',v'') \in E(W) &\iff& (\lambda_{X'}(u'),\lambda_{X''}(v'')) \in E(G) \\
&\iff& (\lambda_{X'}(\phi_{{X_1}X'}(u)), \lambda_{X''}(\phi_{{X_1}X''}(v)) \in E(G) \\
&\iff& (\lambda_{{X_1}}(u) a_{X'}, \lambda_{{X_1}}(v) a_{X''}) \in E(G). \\
\end{array}$$
Suppose $a_{X'}a = h_2a_{X_2}$ and $a_{X''}a = h_3 a_{X_3}$ for $h_2, h_3 \in H$ and $a_{X_2},a_{X_3} \in \mathcal{S}$.  Then
$$\begin{array}{rcl}
(u'\overline{a},v''\overline{a}) &=& (u \overline{h_2} \overline{a_{X_2}}, v \overline{h_3} \overline{a_{X_3}}) \\
&=& (\phi_{{X_1}{X_2}}(u t_{h_2}), \phi_{{X_1}{X_3}}(v t_{h_3}) ).
\end{array}$$
As Proposition \ref{trans} (i)  implies $\bar{a}$ will map distinct elements of $\{Y_X : X \in \Sigma\}$ to distinct elements, we must have $X_2 \ne X_3$, and so
$$\begin{array}{rcl}
(u'\overline{a}, v''\overline{a}) \in E(W) &\iff & (\phi_{{X_1}{X_2}}(u t_{h_2}), \phi_{{X_1}{X_3}}(v t_{h_3}) ) \in E(W) \\
&\iff& (\lambda_{X_2}(\phi_{{X_1}{X_2}}(u t_{h_2})), \lambda_{X_3}(\phi_{{X_1}{X_3}}(vt_{h_3}))) \in E(G) \\
&\iff& (\lambda_{{X_1}}(ut_{h_2})a_{X_2}, \lambda_{{X_1}}(vt_{h_3})a_{X_3}) \in E(G) \\
&\iff& (\lambda_{{X_1}}(u) h_2 a_{X_2}, \lambda_{{X_1}}(v) h_3 a_{X_3}) \in E(G) \\
&\iff& (\lambda_{{X_1}}(u) a_{X'} a, \lambda_{{X_1}}(v) a_{X''} a ) \in E(G) \\
&\iff& (\lambda_{{X_1}}(u) a_{X'},\lambda_{{X_1}}(v) a_{X''}) \in E(G) \\
&\iff& (\lambda_{X'}(u\overline{a_{X'}}),\lambda_{X''}(v\overline{a_{X''}})) \in E(G)\\
&\iff& (\lambda_{X'}(u'),\lambda_{X''}(v'')) \in E(G)\\
&\iff& (u',v'') \in E(W).\\
\end{array}$$
Therefore, $(u'\overline{a},v''\overline{a}) \in E(W)$ when $X' \ne X''$.

If $X'=X''$, then continuing with the above notation,
$$\begin{array}{rcl}
(u'\overline{a},v''\overline{a}) &=& (u \overline{h} \overline{a_{X_2}}, v \overline{h} \overline{a_{X_2}}) \\
&=&  (\phi_{{X_1}{X_2}}(ut_h), \phi_{{X_1}{X_2}}(vt_h)),
\end{array}$$
then the fact that $\phi_{{X_1}{X_2}}$ is a graph isomorphism means this is an edge in $W$ if and only if $(ut_h, vt_h) \in E(B_{X_1})$. When $(u',v'') \in E(W)$, we have $(\phi_{{X_1}X'}(u), \phi_{{X_1}X'}(v)) \in E(B_{X'}) \iff (u,v) \in E(B_{X_1}) \iff (ut_h, vt_h) \in E(B_{X_1}).$ So $(u'\overline{a},v''\overline{a}) \in E(W)$.

From the above we can conclude that $K$ and $\overline{A}$ are subsets of $\aut(W)$, so they generate a subgroup of $\aut(W)$.  This proves the theorem.
\end{proof}

\begin{corollary}
Suppose that $G=Cay(A,S_A)$ is a Cayley graph, $\Sigma$ is a system of blocks for the natural action of $A$ on $G$, and $\{B_X=Cay(C_X,S_X)\}_{X \in \Sigma}$ is a set of pairwise isomorphic Cayley graphs.  Let $\theta_{X_1}: C_{X_1} \rightarrow A^{X_1}$ be a group epimorphism agreeing with a graph epimorphism from $B_{X_1}$ to $G[X_1]$ that induces graph epimorphisms $\lambda_X: B_{X} \rightarrow X$ for all $X \in \Sigma$.
Then the graph $G \circ_\lambda \{B_X\}_{X \in \Sigma}$ is a vertex transitive graph.
\end{corollary}

\begin{proof}
By Proposition \ref{subgp}, $A \circ_{\lambda} \{ C_X \}_{X \in \Sigma}$ is a subgroup of $\aut W$.
As we can interpret the natural action of $A$ on $V(G)$ as right multiplication of $A$ on itself, it follows from Proposition \ref{trans} (ii) that $A\circ_\lambda \{C_X\} _{X\in \Sigma}$ is a subgroup of $\aut W$ that acts transitively on $V(W)$.
\end{proof}

\begin{theorem}\label{main}
Let $W = G \circ_{\lambda} \{B_X\}_{X \in \Sigma}$ be a generalized $X$-join of a Cayley graph $G=Cay(A,S_A)$
with a set of pairwise isomorphic Cayley graphs $\{B_X=Cay(C_X,S_X)\}_{X\in \Sigma }$.
Let $H$ be the stabilizer in $A$ of a fixed $X_1 \in \Sigma$, and let $K_{X_1}$ be the kernel of the group epimorphism $C_{X_1} \rightarrow H$.
Suppose $A = H C_A(H)$, $C_{X_1}=K_{X_1} C_{C_{X_1}}(K_{X_1})$, and either $A = H \times A_1$ for some subgroup $A_1$ of $A$, or $C_{X_1} \simeq K_{X_1} \rtimes H$.  Then $W$ is a Cayley graph.
\end{theorem}

\begin{proof}
Let $J = \{ \hat{k} : k \in K_X \}$ and $\overline{A}$ be as defined in Section 2.  The assumptions here allow us to apply Theorem \ref{reggwps}, so we can conclude that $R=J \overline{A}$ is a subgroup of $\aut W$ that acts regularly on $V(W)$.  Therefore, $W$ is a Cayley graph on $R$.
\end{proof}

We now proceed to determine the connection set for the Cayley graph in Theorem \ref{main}.  Let $e$ be the vertex in $Y_{X_1}$ labelled by the identity element of $C_{X_1}$.  Since $R$ acts regularly on $W$, the set $Y = V(W)$ is equal to $\{ e \cdot r : r \in R \}$.  For each $X \in \Sigma$, we have set bijections between $Y_X$, $C_X$, and $J \overline{H} \, \overline{a_X}$.  So each vertex in $Y$ is labelled by an element of $C_X$ and as $e \cdot r$ for some $r \in J \overline{H} \, \overline{a_X}$.  For subsets $S$ of $C_X$, $X \in \Sigma$, we denote by ${S}^*$ the set of elements of $R$ which correspond to elements of $S$ under these bijections.

With the labelling by elements of $R$ the connection set for this Cayley graph is determined by the elements $r \in R$ for which $(e \cdot 1, e \cdot r)$ is an edge in $W$.  If $e \cdot r$ is the label of a vertex in $Y_{X_1}$, then $r = \hat{\ell} \overline{h}$ for $\ell \in K_X$ and $h \in H$.  If $(e \cdot 1, e \cdot r) \in E(B_{X_1})$, then $r = \hat{\ell} \overline{h}$ corresponds in $C_{X_1}$ to $\ell t_h \in S_{X_1}$.  Therefore, $r \in [{S_{X_1}}]^*$.

If $r$ represents the same vertex as an element of $Y_X$, for $X \ne X_1$, then $r = \hat{\ell} \overline{h} \, \overline{a_X}$ for some $\ell \in K_X$ and $h \in H$.  The vertex of $Y_X$ represented by $e \cdot r$ for this $r$ will be $\phi_{X_1X}(\ell t_h)$.  It follows that the image of this vertex under $\lambda_X$ will be equal to $\lambda_{X_1}(\ell t_h) a_X = h a_X$, and so
$$\begin{array}{rcl}
(e \cdot 1, e \cdot r) \in E(W)&\iff&(\lambda_{X_1}(1_{C_{X_1}}), \lambda_X (\phi_{X_1X}(\ell t_h))) \in E(G)\\
&\iff& (1_A, h a_X) \in E(G) \\
&\iff& h a_X \in S_A \cap Ha_X \\
&\iff& \phi_{X_1X}(\ell t_h) \in \lambda_X^{-1}(S_A \cap Ha_X).
\end{array}$$
Therefore, $r \in [{\lambda_X^{-1}(S_A \cap Ha_X)}]^*$, and so we can conclude that $W = Cay( R, S_R)$, where $R = J \overline{A}$ and
$$ S_R =[{S_{X_1}}]^* \cup \big( \bigcup_{X \ne X_1} [{\lambda_X^{-1}(S_A \cap Ha_X)}]^* \big). $$


\begin{example} {\rm
Let $A = \langle x,y : x^3=y^2=1, xy=yx^2\rangle \simeq D_6$ and $G = Cay(A,\{x,x^2,y\})$.  Let $C_X = \langle a,b : a^3=b^3=1,ab=ba \rangle \simeq C_3 \times C_3$ and $B_X = Cay(C_X, \{a,a^2,b,b^2\} )$.  When $H=\langle x \rangle$, we get $H = X$ and $X' = Hy$, and the induced subgraphs $G[X] = Cay(H,\{x,x^2\})$ and $G[X'] = G[X]y$ are disjoint triangles.

\begin{figure}[htb]
\begin{center}
\begin{tikzpicture}
\filldraw (10,1) circle [radius=0.04] node  [left] {$y$} -- (11,2) circle [radius=0.04] node  [right] {$x^2y$} -- (12,1) circle [radius=0.04] node  [right] {$xy$} -- (10,1) circle [radius=0.04] node  [left] {$y$} -- (10,3) circle [radius=0.04] node  [left] {$e$} -- (11,4) circle [radius=0.04] node  [above] {$x$} -- (12,3) circle [radius=0.04] node  [right] {$x^2$} -- (10,3) circle [radius=0.04] node  [left] {$e$};
\filldraw (11,2) circle [radius=0.04] node  [right] {$x^2y$} -- (11,4) circle [radius=0.04] node  [above] {$x$};
\filldraw (12,1) circle [radius=0.04] node  [right] {$xy$} -- (12,3) circle [radius=0.04] node  [right] {$x^2$};
\end{tikzpicture}
\subfloat{$G$}\;\;\;\;
\begin{tikzpicture}
\filldraw (0,0) circle [radius=0.04] node  [below] {$b^2$} -- (1,1) circle [radius=0.04] node  [left] {$ab^2$} -- (2,0) circle [radius=0.04] node  [below] {$a^2b^2$} -- (0,0) circle [radius=0.04] node  [below] {$b^2$} -- (0,2) circle [radius=0.04] node  [left] {$b$} -- (1,3) circle [radius=0.04] node  [left] {$ab$} -- (2,2) circle [radius=0.04] node  [right] {$a^2b$} -- (0,2) circle [radius=0.04] node  [left] {$b$} -- (0,4) circle [radius=0.04] node  [left] {$1$} -- (1,5) circle [radius=0.04] node  [above] {$a$} -- (2,4) circle [radius=0.04] node  [right] {$a^2$} -- (0,4) circle [radius=0.04] node  [left] {$1$};
\filldraw (1,1) circle [radius=0.04] node [left] {$ab^2$} -- (1,3) circle [radius=0.04] node  [left]{$ab$} -- (1,5) circle [radius=0.04] node  [above] {$a$};
\filldraw (2,0) circle [radius=0.04] node [below] {$a^2b^2$} -- (2,2) circle [radius=0.04] node  [right] {$a^2b$} -- (2,4) circle [radius=0.04] node [right] {$a^2$};
\filldraw (0,4) circle [radius=0.04] node [left] {$1$} -- (0,0) circle [radius=0.04] node  [below]{$b^2$};
\filldraw (1,5) circle [radius=0.04] node [above] {$a$} -- (1,1) circle [radius=0.04] node  [left]{$ab^2$};
\filldraw (2,4) circle [radius=0.04] node [right] {$a^2$} -- (2,0) circle [radius=0.04] node  [below]{$a^2b^2$};

\draw [fill] (1,5) circle [radius=0.04] node [above] {$a$};
\draw [fill] (1,1) circle [radius=0.04] node [left] {$ab^2$};
\path (1,5) edge [bend right] (1,1);
\draw [fill] (2,4) circle [radius=0.04] node [right] {$a^2$};
\draw [fill] (2,0) circle [radius=0.04] node [below] {$a^2b^2$};
\path (2,4) edge [bend right] (2,0);
\draw [fill] (0,4) circle [radius=0.04] node [left] {$1$};
\draw [fill] (0,0) circle [radius=0.04] node [below] {$b^2$};
\path (0,4) edge [bend right] (0,0);
 \end{tikzpicture}
\subfloat{$B_X$}\;\;\;\;\;
\begin{tikzpicture}
\filldraw (6,-3) circle [radius=0.04] node  [below] {$b^2\bar{y}$} -- (7,-2) circle [radius=0.04] node  [left] {$ab^2\bar{y}$} -- (8,-3) circle [radius=0.04] node  [below] {$a^2b^2\bar{y}$} -- (6,-3) circle [radius=0.04] node  [below] {$b^2\bar{y}$} -- (6,-1) circle [radius=0.04] node [left] {$b\bar{y}$} -- (7,0) circle [radius=0.04] node  [right] {$ab\bar{y}$} -- (8,-1) circle [radius=0.04] node  [right] {$a^2b\bar{y}$} -- (6,-1) circle [radius=0.04] node  [left] {$b\bar{y}$} -- (6,1) circle [radius=0.04] node  [left] {$1\bar{y}$} -- (7,2)
circle [radius=0.04] node  [above] {$a\bar{y}$} -- (8,1) circle [radius=0.04] node  [right] {$a^2\bar{y}$} -- (6,1) circle [radius=0.04] node  [left] {$1\bar{y}$};
\filldraw (7,-2) circle [radius=0.04] node [left] {$ab^2\bar{y}$} -- (7,0) circle [radius=0.04] node  [right] {$ab\bar{y}$} -- (7,2) circle [radius=0.04] node  [above] {$a\bar{y}$};
\filldraw (8,-3) circle [radius=0.04] node [below] {$a^2b^2\bar{y}$} -- (8,-1) circle [radius=0.04] node  [right] {$a^2b\bar{y}$} -- (8,1) circle [radius=0.04] node [right] {$a^2\bar{y}$};
\filldraw (6,1) circle [radius=0.04] node [left] {$1\bar{y}$} -- (6,-3) circle [radius=0.04] node  [below]{$b^2\bar{y}$};
\filldraw (7,2) circle [radius=0.04] node [above] {$a\bar{y}$} -- (7,-2) circle [radius=0.04] node  [left]{$ab^2\bar{y}$};
\filldraw (8,1) circle [radius=0.04] node [right] {$a^2\bar{y}$} -- (8,-3) circle [radius=0.04] node  [below]{$a^2b^2\bar{y}$};

\draw [fill] (6,1) circle [radius=0.04] node [left] {$1\bar{y}$};
\draw [fill] (6,-3) circle [radius=0.04] node [below] {$b^2\bar{y}$};
\path (6,1) edge [bend right] (6,-3);
\draw [fill] (7,2) circle [radius=0.04] node [above] {$a\bar{y}$};
\draw [fill] (7,-2) circle [radius=0.04] node [left] {$ab^2\bar{y}$};
\path (7,2) edge [bend right] (7,-2);
\draw [fill] (8,1) circle [radius=0.04] node [right] {$a^2\bar{y}$};
\draw [fill] (8,-3) circle [radius=0.04] node [below] {$a^2b^2\bar{y}$};
\path (8,1) edge [bend right] (8,-3);
\end{tikzpicture}
\subfloat{$B_{X'}$}
\end{center}
\caption{Graph $G$ and set of graphs $\{B_X\}_{X\in\Sigma}$}
\end{figure}
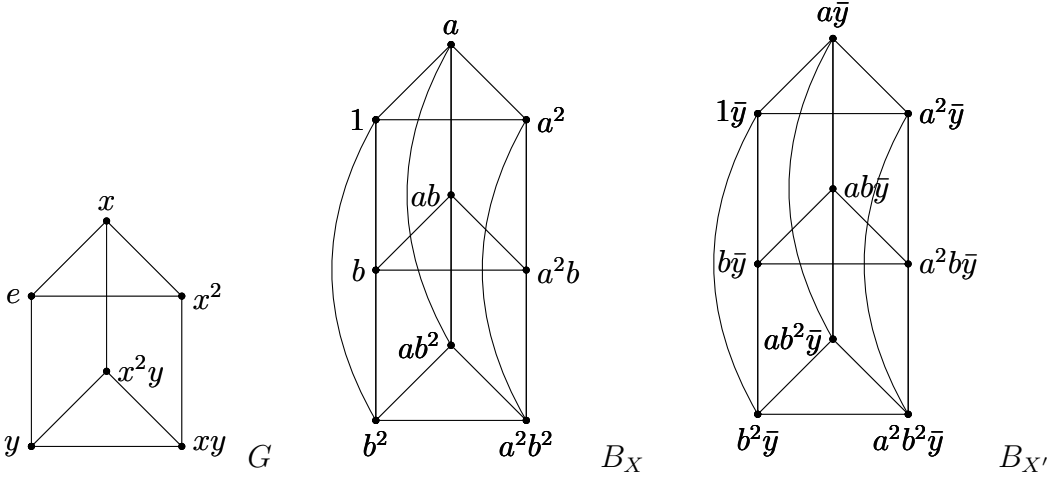
Right multiplication by $y$ is a map $\varepsilon_{XX'}: X \rightarrow X'$.  Let $Y_X=C_X$ and let $Y_{X'} = Y_X \bar{y}$ be a copy of $C_X$, with right multiplication by $\bar{y} \in \bar{A}$ identified with $\phi_{XX'}: Y_X \rightarrow Y_{X'}$, so $Y_{X'} = \{ c \bar{y} : c \in Y_X \}$.

Let $\theta_X: C_X \rightarrow H$ be the group epimorphism that maps $a \mapsto x$ and $b \mapsto 1$, so we have $K_X = \langle b \rangle$.  We can identify $C_{X'}$ with $C_X^{\bar{y}}$, and then the right multiplications of $C_X$ on $Y_X$ and of $C_{X'}$ on $Y_{X'}$ will act as graph isomorphisms.

Let $\lambda_X: B_X \rightarrow G[X]$ be equal to $\theta_X$ as a map of sets, and let $\lambda_{X'} = \varepsilon_{XX'} \circ \lambda_X \circ \phi_{XX'}^{-1}$.  Then $\lambda_X$ and $\lambda_{X'}$ are graph epimorphisms, and we can define the generalized $X$-join $W=G \circ_{\lambda} \{ B_X, B_{X'} \}$.

\begin{figure}[htb]
\begin{center}
\begin{tikzpicture}
\filldraw (0,0) circle [radius=0.04] node  [left] {$\widehat{b^2}\bar{e}$} -- (1,1) circle [radius=0.04] node  [left] {$\widehat{b^2}\bar{x}$} -- (2,-1) circle [radius=0.04] node  [below] {$\widehat{b^2}\overline{x^2}$} -- (0,0) circle [radius=0.04] node  [left] {$\widehat{b^2}\bar{e}$} -- (0,2.5) circle [radius=0.04] node  [left] {$\hat{b}\bar{e}$} -- (1,3.5) circle [radius=0.04] node  [left] {$\hat{b}\bar{x}$} -- (2,2) circle [radius=0.04] node  [below] {$\hat{b}\overline{x^2}$} -- (0,2.5) circle [radius=0.04] node  [left] {$\hat{b}\bar{e}$} -- (0,4.5) circle [radius=0.04] node  [left] {$\bar{e}$} -- (1,5.5) circle [radius=0.04] node  [above] {$\bar{x}$} -- (2,4) circle [radius=0.04] node  [below] {$\overline{x^2}$} -- (0,4.5) circle [radius=0.04] node  [left] {$\bar{e}$};
\filldraw (1,1) circle [radius=0.04] node [left] {$\widehat{b^2}\bar{x}$} -- (1,3.5) circle [radius=0.04] node  [left] {$\hat{b}\bar{x}$} -- (1,5.5) circle [radius=0.04] node  [above] {$\bar{x}$};
\filldraw (2,-1) circle [radius=0.04] node [below] {$\widehat{b^2}\overline{x^2}$} -- (2,2) circle [radius=0.04] node  [below] {$\hat{b}\overline{x^2}$} -- (2,4) circle [radius=0.04] node [below] {$\overline{x^2}$};
\filldraw (5,-2) circle [radius=0.04] node  [below] {$\overline{x^2y}$} -- (7,0) circle [radius=0.04] node
[right] {$\hat{b}\overline{x^2y}$} -- (9,-1) circle [radius=0.04] node  [right] {$\widehat{b^2}\overline{x^2y}$} -- (5,-2) circle [radius=0.04] node  [below] {$\overline{x^2y}$} -- (5,1) circle [radius=0.04] node [below] {$\overline{xy}$} -- (7,2) circle [radius=0.04] node  [right] {$\hat{b}\overline{xy}$} -- (9,1) circle [radius=0.04] node  [right] {$\widehat{b^2}\overline{xy}$} -- (5,1) circle [radius=0.04] node  [below] {$\overline{xy}$} -- (5,3) circle [radius=0.04] node  [above] {$\bar{y}$} -- (7,4) circle [radius=0.04] node  [above] {$\hat{b}\bar{y}$} -- (9,3) circle [radius=0.04] node  [right] {$\widehat{b^2}\bar{y}$} -- (5,3) circle [radius=0.04] node  [above] {$\bar{y}$};
\filldraw (7,0) circle [radius=0.04] node [right] {$\hat{b}\overline{x^2y}$} -- (7,2) circle [radius=0.04] node  [right] {$\hat{b}\overline{xy}$} -- (7,4) circle [radius=0.04] node  [above] {$\hat{b}\bar{y}$};
\filldraw (9,-1) circle [radius=0.04] node [right] {$\widehat{b^2}\overline{x^2y}$} -- (9,1) circle [radius=0.04] node  [right] {$\widehat{b^2}\overline{xy}$} -- (9,3) circle [radius=0.04] node [right] {$\widehat{b^2}\bar{y}$};
\filldraw (1,5.5) circle [radius=0.04] node  [above] {$\bar{x}$} -- (7,4) circle [radius=0.04] node  [above] {$\hat{b}\bar{y}$} -- (0,4.5) circle [radius=0.04] node  [left] {$\bar{e}$} -- (5,3) circle [radius=0.04] node  [above] {$\bar{y}$} -- (2,4) circle [radius=0.04] node  [below] {$\overline{x^2}$} -- (9,3) circle [radius=0.04] node  [right] {$\widehat{b^2}\bar{y}$} -- (1,5.5) circle [radius=0.04] node  [above] {$\bar{x}$};
\filldraw (1,5.5) circle [radius=0.04] node  [above] {$\bar{x}$} -- (5,3) circle [radius=0.04] node  [above] {$\bar{y}$};
\filldraw (2,4) circle [radius=0.04] node  [below] {$\overline{x^2}$} -- (7,4) circle [radius=0.04] node  [above] {$\hat{b}\overline{y}$};
\filldraw (0,4.5) circle [radius=0.04] node  [left] {$\bar{e}$} -- (9,3) circle [radius=0.04] node  [right] {$\widehat{b^2}\bar{y}$};
\filldraw (1,3.5) circle [radius=0.04] node  [left] {$\hat{b}\bar{x}$} -- (7,2) circle [radius=0.04] node  [right] {$\hat{b}\overline{xy}$} -- (0,2.5) circle [radius=0.04] node  [left] {$\hat{b}\bar{e}$} -- (5,1) circle [radius=0.04] node  [below] {$\overline{xy}$} -- (2,2) circle [radius=0.04] node  [below] {$\hat{b}\overline{x^2}$} -- (9,1) circle [radius=0.04] node  [right] {$\widehat{b^2}\overline{xy}$} -- (1,3.5) circle [radius=0.04] node  [left] {$\hat{b}\bar{x}$};
\filldraw (1,3.5) circle [radius=0.04] node  [left] {$\hat{b}\bar{x}$} -- (5,1) circle [radius=0.04] node  [below] {$\overline{xy}$};
\filldraw (2,2) circle [radius=0.04] node  [below] {$\hat{b}\overline{x^2}$} -- (7,2) circle [radius=0.04] node  [right] {$\hat{b}\overline{xy}$};
\filldraw (0,2.5) circle [radius=0.04] node  [left] {$\hat{b}\bar{e}$} -- (9,1) circle [radius=0.04] node  [right] {$\widehat{b^2}\overline{xy}$};
\filldraw (1,1) circle [radius=0.04] node  [left] {$\widehat{b^2}\bar{x}$} -- (7,0) circle [radius=0.04] node  [right] {$\hat{b}\overline{x^2y}$} -- (0,0) circle [radius=0.04] node  [left] {$\widehat{b^2}\bar{e}$} -- (5,-2) circle [radius=0.04] node  [below] {$\overline{x^2y}$} -- (2,-1) circle [radius=0.04] node  [below] {$\widehat{b^2}\overline{x^2}$} -- (9,-1) circle [radius=0.04] node  [right] {$\widehat{b^2}\overline{x^2y}$} -- (1,1) circle [radius=0.04] node  [left] {$\hat{b^2}\bar{x}$};
\filldraw (1,1) circle [radius=0.04] node  [left] {$\widehat{b^2}\bar{x}$} -- (5,-2) circle [radius=0.04] node  [below] {$\overline{x^2y}$};
\filldraw (2,-1) circle [radius=0.04] node  [below] {$\widehat{b^2}\overline{x^2}$} -- (7,0) circle [radius=0.04] node  [right] {$\hat{b}\overline{x^2y}$};
\filldraw (0,0) circle [radius=0.04] node  [left] {$\widehat{b^2}\bar{e}$} -- (9,-1) circle [radius=0.04] node  [right] {$\widehat{b^2}\overline{x^2y}$};

\draw [fill] (1,5.5) circle [radius=0.04] node [above] {$\bar{x}$};
\draw [fill] (1,1) circle [radius=0.04] node [left] {$\widehat{b^2}\bar{x}$};
\path (1,5.5) edge [bend right] (1,1);
\draw [fill] (2,4) circle [radius=0.04] node [below] {$\overline{x^2}$};
\draw [fill] (2,-1) circle [radius=0.04] node [below] {$\widehat{b^2}\overline{x^2}$};
\path (2,4) edge [bend left] (2,-1);
\draw [fill] (0,4.5) circle [radius=0.04] node [left] {$\bar{e}$};
\draw [fill] (0,0) circle [radius=0.04] node [left] {$\widehat{b^2}\bar{e}$};
\path (0,4.5) edge [bend right] (0,0);
\draw [fill] (5,3) circle [radius=0.04] node [above] {$\bar{y}$};
\draw [fill] (5,-2) circle [radius=0.04] node [below] {$\overline{x^2y}$};
\path (5,3) edge [bend right] (5,-2);
\draw [fill] (7,4) circle [radius=0.04] node [above] {$\hat{b}\bar{y}$};
\draw [fill] (7,0) circle [radius=0.04] node [right] {$\hat{b}\overline{x^2y}$};
\path (7,4) edge [bend right] (7,0);
\draw [fill] (9,3) circle [radius=0.04] node [right] {$\widehat{b^2}\bar{y}$};
\draw [fill] (9,-1) circle [radius=0.04] node [right] {$\widehat{b^2}\overline{x^2y}$};
\path (9,3) edge [bend right] (9,-1);
\end{tikzpicture}
\end{center}
\caption{Graph $W$}
\end{figure}

We have that $J = \{ (1,1), (b,b^{\bar{y}}), (b^2, b^{2\bar{y}}) \}$. We can define $\bar{y}$ so that $\{\bar{1},\bar{y}\}$ is a group, and also a transversal for $T_H$ in $\bar{A}$, which will then be group isomorphic to $A$. Indeed, setting $\bar{x} = (c \mapsto ca, c\bar{y} \mapsto ca^2 \bar{y})$ and $\overline{x^2} = (c \mapsto ca^2, c\bar{y} \mapsto ca \bar{y})$ will result in $\bar{A} = \langle \bar{x}, \bar{y} \rangle$ being isomorphic to $A$.  As the elements of $\bar{A}$ commute with $J$, $J \bar{A}$ will be a group isomorphic to $D_6 \times C_3$ that acts regularly on $V(W)=Y$.  So $W$ is a Cayley graph.

The connection set for this Cayley graph will be the union of $[{S_X}]^*= \{ \bar{x}, \overline{x^2}, \hat{b}, \widehat{b^2} \}$ and $[\lambda_{X'}^{-1}(y)]^* = \{ \bar{y}, \hat{b}\bar{y}, \widehat{b^2}\bar{y} \}$.
}
\end{example}

\begin{example} {\rm
Let $A = C_2 \times C_2 \times C_2 = \langle a \rangle \times \langle b \rangle \times \langle c \rangle$ be the elementary abelian group with identity $e$.  Let $G$ be the Cayley graph $Cay(A,\{a,b,c\})$.  Let $H=\langle a \rangle \times \langle b \rangle$, and let $\Sigma = \{X=H, X'=Hc \}$.  Then $G[X] = Cay(H, \{a,b\})$ is graph isomorphic to $G[X']$ via the map $\varepsilon_{XX'}$ that is defined by right multiplication by $c$.

Let $C_X \simeq Q_8 = \{\pm 1, \pm i, \pm j, \pm k \}$ be the quaternion group of order $8$, with the usual conventions $i^2=j^2=k^2=-1$ and $k = ij =-ji$.  We have a group epimorphism $\theta_X: C_X \rightarrow H$ given by $\theta_X(i) = a$ and $\theta_X(j)=b$ with kernel $K_X = \{\pm 1\}$.  If $B_X = Cay(Q_8,\{\pm i, \pm j \})$, then $\theta_X$ induces a graph epimorphism $\lambda_X:B_X \rightarrow G[X]$ for which $|\lambda_X^{-1}(g)|=2$ for all $g \in A$.   Define $B_{X'}$ to be an isomorphic copy of $B_X$, and let $\phi_{XX'} : B_X \rightarrow B_{X'}$ be a fixed graph isomorphism.   Define $\lambda_{X'}: B_{X'} \rightarrow G[X']$ by $\lambda_{X'} = \varepsilon_{XX'} \circ \lambda_X \circ \phi_{XX'}^{-1}$.  Then we have a generalized $X$-join of Cayley graphs $G \circ_{\lambda} \{ B_X, B_{X'} \}$.  The assumptions of the main theorem are satisfied because $A = H \times C_2$ and $K_X = Z(C_X)$.

In order to understand the group $R$, we need to first understand the groups $J$ and $\bar{A}$.  Let $Y_X$ and $Y_{X'}$ be the vertices of $B_X$ and $B_{X'}$.  First, it is clear that $J = \{ \hat{1}, \widehat{-1} \}$.  For every element $\bar{a} \in \bar{A}$ and every component $Y_{X'}$ of $Y$, there will be $|K_{X}|$ choices for the definition of the restriction of $\bar{a}$ to $Y_{X'}$, and we must be able to choose these definitions so that $\bar{A}^2 \subseteq J\bar{A}$ and $\bar{A}J = J \bar{A}$.

We will assume $\bar{e}$ is the identity map on $Y$ when $e$ is the identity of $A$.  Since $\varepsilon_{XX'}$ agrees with right multiplication by $c$  we can assume $\phi_{XX'} = \overline{c}$.  For the remaining six elements of $\bar{A}$, the choices given in the following table make $J \bar{A}$ into a group:

$$\begin{array}{r|c|c r|c|c}
\bar{g} \in \bar{A} & x \in Y_X \mapsto \quad & x\bar{c} \in Y_{X'} \mapsto \\ \hline
\bar{e} & x & x\bar{c} \\
\bar{a} & xi & xi \bar{c} \\
\bar{b} & xj & xj \bar{c} \\
\overline{ab} & xk & xk \bar{c} \\
\bar{c} & x \bar{c} & x \\
\overline{ac} & xi \bar{c} & xi \\
\overline{bc} & xj \bar{c} & xj \\
\overline{abc} & xk \bar{c} & xk
\end{array} $$

With elements of $\bar{A}$ defined this way, $\bar{A}$ is not a group, but nevertheless $\bar{A}^2 \subseteq J \bar{A}$ and the elements of $\bar{A}$ commute with $\widehat{-1}$, which maps $y \rightarrow -y$ for all $y \in Y$. So $R= J \bar{A}$ is a group, and it is easy to see from the table that it will act regularly on the vertices of $W = G \circ_{\lambda} \{B_X, B_{X'} \}$.
To calculate the connection set, let $1_X$ be the vertex labelled by the identity of $C_X$. For $y \in Y_X$, $(1_X,y) \in E(W)$ for $y \in \{ \pm i, \pm j \}$.  The elements of $J \bar{A}$ that map $1_X$ to these vertices are the ones in $J \bar{a} \cup J \bar{b}$.  These are the elements of $J \bar{A}$ that map $1_X$ to an element of $[{S_X}]^*$.  For $y \bar{c} \in Y_{X'}$, $(1_X, y \bar{c}) \in E(W)$ if and only if $(1_A,\lambda_X(y)c) \in S_A$, which is equivalent to $y \bar{c} \in \{ \pm 1 \bar{c} \}$.  The elements of $J \bar{A}$ that map $1_X$ to these vertices are the elements of $J \bar{c}$.   Therefore,
$$ {S_R} = J \bar{a} \cup J \bar{b} \cup J \bar{c} =[{S_X}]^* \cup [{\lambda_{X'}^{-1}(S_A \cap Hc)}]^*. $$
}
\end{example}

\begin{corollary}
Let $G=Cay(A,S_A)$ be a Cayley graph and $\{B_x=Cay(C_x,S_x)\}_{x \in A}$ be a set of pairwise isomorphic Cayley graphs. Then the
$G$-join $G[B_x]_{x\in A}$ is a Cayley graph.

\end{corollary}
\begin{proof}
With the notation of Theorem \ref{main} if $X_1=\{e\}$, then $H=\{e\}$ and $K_{X_1}=C_{X_1}$.
Clearly, $A = H C_A(H)$, $C_{X_1}=K_{X_1} C_{C_{X_1}}(K_{X_1})$, and  $A = H \times A$.  So  it follows from Theorem \ref{main}
that the $G$-join $G[B_x]_{x\in A}$ is a Cayley graph.
\end{proof}

{\small

\medskip
\noindent {\bf Acknowledgement.}
The third author would like to thank Dr. Herman and the University of Regina for their hospitality during which much of this work was completed.
}

\noindent Department of Mathematics and Statistics, University of Regina, 3737 Wascana Parkway, Regina, SK, S4S 0A2  Canada\\
\noindent {\tt email: Allen.Herman@uregina.ca}

\smallskip
\noindent Department of Pure Mathematics,  Faculty of Mathematics and Statistics,\\ University of Isfahan,
P.O. Box: 81746-73441, Isfahan, Iran \\
\noindent {\tt email: bagherian@sci.ui.ac.ir}

\smallskip
\noindent Department of Pure Mathematics,  Faculty of Mathematics and Statistics,\\ University of Isfahan,
P.O. Box: 81746-73441, Isfahan, Iran \\
\noindent {\tt email: h.memarzadeh.762@sci.ui.ac.ir}

\end{document}